\documentclass[reqno,final]{amsart}
\usepackage{natbib}  
\usepackage{fancyhdr} 
\usepackage{color} 
\usepackage{hyperref} 
\usepackage{graphicx} 

\usepackage{amssymb,mathrsfs}
\usepackage{graphicx}

\renewcommand{\cite}{\citet}

\theoremstyle{plain}
\newtheorem{theorem}{Theorem}[section]                                          
                          
\newtheorem{lemma}[theorem]{Lemma}
\newtheorem{corollary}[theorem]{Corollary}

\theoremstyle{definition}
\newtheorem{definition}[theorem]{Definition}
\theoremstyle{remark}
\newtheorem{remark}[theorem]{Remark}

\makeatletter \@addtoreset{equation}{section} \makeatother

	\def\N{\mathbb{N}}
	\def\d{\sqrt{2}}
	\def\R{\mathbb{R}}
	\def\H{\mathscr{H}}
	\def\L{\mathscr{L}}
	
\newcommand{\fracPart}[1]{[\!|#1|\!]_1}

\newcommand{\st}[2]{\mathbf{s}_{#1\!}(#2)}
\newcommand{\ct}[2]{\mathbf{c}_{#1\!}(#2)}

\begin{document}

\title[Noise models on the circumference]{Is the Brownian bridge a good noise model 
on the boundary of a circle?}\thanks{The final publication is available at  {http://www.ism.ac.jp/editsec/aism/}}

\author[G.~Aletti]{Giacomo Aletti}
\author[M.~Ruffini]{Matteo Ruffini}

\address{G.~Aletti\\
              ADAMSS Center \& Dept. of Mathematics ``Federico Enriques''\\ 
              Universit\`a degli Studi di Milano\\
              Via Saldini 50, 20131 Milano} 

\email{giacomo.aletti@unimi.it}

\address{M.~Ruffini\\                          ToolsGroup Spain \\
C/Diputación, 303, \'Atico \\
08009 Barcelona,
Spain} 

\email{mruffini@toolsgroup.com}


\begin{abstract}
In this paper we study periodical stochastic processes, and we define the conditions 
that are needed by a model to be a good noise model on the circumference. 
The classes of processes that fit the required conditions are studied together with their expansion 
in random Fourier series in order to provide results about their path regularity.
Finally, we discuss a simple and flexible parametric model with prescribed
regularity that is used in applications, and we prove the asymptotic properties
of the maximum likelihood estimates of model parameters. 
\end{abstract}



\maketitle

\section{Introduction}
\subsection{Literature review}

Modeling the random boundaries of star-shaped planar objects 
is a topic that is receiving an increasing interest in recent times. 
Some examples can be found in neurology (see \cite{HoboltNeurology} and the references therein), geography \cite{Burrough}, stereology (see \cite{HoboltShape,HoboltStereology} and the references therein), fractal geometry analysis (see \cite{DiogAletti} and the references therein). 

A common way to model such a phenomena is to model the radius-vector 
function as a periodic stochastic process from an interval to $\R$; a detailed geometric description of this model is provided in \cite{Lieshout}.
In such a framework, the radius of the star-shaped planar object is a periodic and continuous function, as a function of the independent variable, representing the angle.

A very standard and well known model that apparently suits these needs is the Brownian Bridge.
The Brownian bridge is a universally known model, used in several
areas of applied mathematical science. As only an example, the recent publication of 
\cite{Kroese_HB,Manganaro} and the reference therein provide a huge relevant literature,
while \cite{Bass_SP} provide a theoretical analysis of such a process.
The main aspect of the Brownian bridge is its periodicity, that makes this process
a good model for a noise on the finite domain $[0,1]$. 
On the other hand, a deficiency of this model is its non-stationarity,
which is almost a must when one models pure noise. This is due to the fact that Brownian bridge
is assumed to be $0$ at $t=0$. 

A second approach that is being obtaining success in recent time, is to exploit the asymptotic results of the random Fourier series 
to provide general models for the boundaries of star-shaped objects.
In \cite{HoboltShape}  the authors propose a parametric 
random Fourier series model (called generalized $p$-order model) to describe the border of random 
planar star-shaped objects in terms of normalized radius-vector function. 
Again, in \cite{HoboltShape}, the authors also provided results 
about sample path regularity, and an expression for the maximum likelihood function for the model parameters, 
even if there are not asymptotic results about these estimators.

\subsection{Overview and insights of the paper}
In this paper we deal with the second approach, using random Fourier series as a flexible modeling tool, finding interesting properties of the 
studied processes thanks to the standard representation they provide.

\bigskip

First, we define the theoretical conditions that are needed by a process to be a good noise model on the boundary of a circle, 
admitting models with a fixed zero value in the origin only as 
the conditioning of such a process, as the result of a selective sampling.
More precisely, two classes of processes are considered:
\begin{itemize}
\item $\H$, the set of Gaussian, stationary, $[0,1]$-periodic processes;
\item $\H_0$, the set of processes generated by a process in $\H$ conditioned to be $0$ when $t=0$.
\end{itemize}

Then, we remark that the Brownian Bridge is not contained in $\H_0$. 
Furthermore, we find a standard Fourier decomposition for a process 
$ \{x_t\}_{t\in[0,1]}$ in $\H$, 
by expressing its covariance function $C(s,t)$ as  
\(
C(s,t) = 
c_0^2+2 \sum_{k=1}^{\infty}{c_k^2\cos(2k\pi(s- t))}.
\)
Thanks to Karhunen-Lo\`eve's theorem, the process $ \{x_t\}_{t\in[0,1]}$ may be represented as
$$
x_t = c_0 Y'_0 + \sum_{k=1}^{\infty}c_k
(Y_k\st{k}{t} + Y'_k\ct{k}{t})
$$ 
where
$
\{Y_k\}_{k \geq 1}$ and $\{Y'_k\}_{k \geq 0}$  are two independent sequences of independent standard Gaussian variables. 

\bigskip

As a consequence, we prove that:
\begin{itemize}
\item  the random Fourier series expansion of a process in $\H_0$ shares the same asymptotic behavior for the spectrum with its generator in $\H$;
\item  the path regularity of a process in $\H_0$ depends on path regularity of its generator in $\H$; in particular we show that 
   the regularity properties of the trajectories of a process in $\H $ and of its generated process in $\H_0$ have the same lower bound in terms of H\"{o}lder  regularity;
  \item the path regularity of a process in $\H$ (and of its generated process in $\H_0$) can be deduced by the Fourier coefficients of the generator process
  covariance function, looking at their decrease rate. In particular, we show that, for any $0 < \alpha \leq 1$,  
  $$
  c_k^2 = O({1}/{k^{1+2m+\alpha}})
  \qquad \Longrightarrow \qquad   \{x_t\}_{t\in[0,1]} \in C^{m, \beta} ([0,1]) , \quad \text{ with }
  \beta < {\alpha}/{2},
  $$
where $C^{m, \beta} ([0,1])$ is the H\"{o}lder space 
of the functions on $[0,1]$ having continuous derivatives 
up to order $m$ and such that the $m^{\text{th}}$-derivative is H\"{o}lder continuous with exponent $\beta$.  
\end{itemize}

\begin{figure}
\begin{center}
\includegraphics[width=0.48\textwidth]{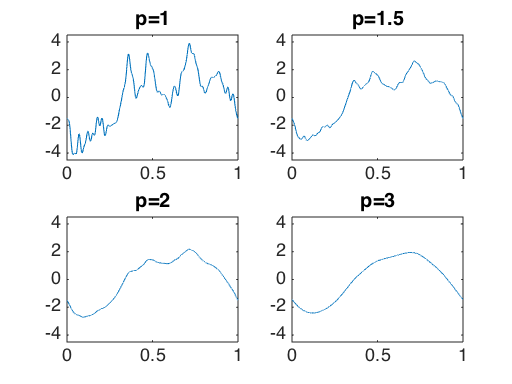} 
\caption{Parametric model in $\H$, where path regularity is determined by the parameter $p$.}\label{fig:0}
\end{center}
\end{figure}
\bigskip

Finally, as in \cite{HoboltShape}, we discuss a general and flexible parametric model in $\H$:
\begin{equation*}
x_t = \sum_{k=1}^{N}\frac{a}{k^{p}}(Y_k\sin(2k\pi t) + Y'_k\cos(2k\pi t)),
\end{equation*} 
together with the generated model in $\H_0$. We underline that
\begin{itemize}
\item  these models provide a very easy way to represent stochastic processes in computer memory, 
where only a finite number of coefficient may be stored. In addition, the representation of the first one 
is built on a finite dimensional subspace made by only trigonometric functions;
\item  the path regularity of $ \{x_t\}_{t\in[0,1]}$ is determined by its parameters, see Figure~\ref{fig:0};
  \item we provide maximum likelihood estimates for the first model, together with asymptotic properties of the estimators.  
\end{itemize}

Summing up, these models can be very useful in the applications: on one hand they might shape 
particular characteristics of the observed phenomena, allowing on the other hand properties similar to the 
Brownian bridge when these are needed, but with a stronger theoretical support.

\subsection{Structure of the paper}

In Section~\ref{sec:2s} we define the two classes of interest, $\H$ and the space of conditioned processes $\H_0$, studying their properties and analyzing their random Fourier series expansion.
In Section~\ref{sec:4s} we show the connection between the spectrum of the processes 
in $\H_0$ with respect to their generators in $\H$.
In Section~\ref{sec:5s} it is proven that also the path regularity is maintained for such
couples, as a consequence of Kolmogorov's continuity conditions and a result of Boas.
In Section~\ref{sec:6s} simple parametric models in $\H$ and $\H_0$ are presented, together with the properties 
of the maximum likelihood estimators for the parameters.

\subsection*{Summary of notations}
The variables $s,t,\ldots$ relate to time variables, and will often belong
to $[0,1]$. We denote by $\{x_t\}_{t\in[0,1]},
\{y_t\}_{t\in[0,1]}, \ldots$ stochastic adapted process
defined on a given filtered space $(\Omega, \mathcal{F}, \{\mathcal{F}_t\}_{t\in[0,1]}, \mathbb{P})$,
while $ \{ X_n\}_{n\geq 1}$, $\{ Y_n \}_{n\geq 1}$, $\{ Z_n\}_{n\geq 1},\ldots$ are sequences of random variables. 
$C(s,t)$ 
is a  positive semidefinite function (it will be the correlation function of a stochastic process).
When a process has stationary increments, its covariance function will often be replaced by 
the associated covariogram function $\tilde{C}(t-s)=C(s,t)$. 
The sequence $\{e_k(t)\}_{k\geq 0}$
denotes a sequence of orthogonal function on $L^2([0,1])$. 
Finally, we denote by $\fracPart{t}$ the fractional part of the real number $t$, that is the sawtooth wave
defined by the formula
$\fracPart{t} = t - \mathrm{floor}(t)$.

\section{Preliminaries and Karhunen-Lo\`eve's decomposition theorem}\label{sec:2s}
In this section we recall some basic results from Gaussian processes theory. The first theorem we need is the Karhunen-Lo\`eve's decomposition theorem (see \cite{Karhunen}), that states what follows.
\begin{theorem}[Karhunen-Lo\`eve]\label{teo:KL}
Let $\{x_t\}_{t\in [0,1]} $, such that
$
E[x_t] \equiv 0,$
and
$
Cov(x_t,x_s) = C(t,s),
$
continuous in both variables. Then 
$
x_t = \sum_{k=1}^{\infty}Z_k\,e_k(t),
$
where 
\begin{itemize}
\item the functions
$\{e_k(\cdot)\}_{k\geq 1}$ are the eigenfunctions of the following integral operator from  $L^2[0,1]$ in itself
\begin{equation}\label{Merc_operator}
f\in L^{2}[0,1]\longrightarrow g(t) = \int_{0}^{1}{C(t,\tau)f(\tau)d\tau},
\end{equation}
and $\{e_k(\cdot)\}_{k\geq 1}$ form an orthonormal basis for the space spanned by the eigenfunctions 
corresponding to nonzero eigenvalues;
\item 
the random variables $Z_1, Z_2, \ldots$ are given by
$
Z_k = \int_{0}^{1}x_te_k(t)dt
$
and form a zero-mean orthogonal system (i.e., $  E(Z_kZ_j) = 0$ for $k\neq j$) with variance $\lambda_k^2$, 
where $\lambda_k$ is the eigenvalue corresponding to the eigenfunction $e_k(\cdot)$.
\end{itemize}
The series $ \sum_{k=1}^{\infty}Z_ke_k(t) $ converges in mean square to $x_t$, uniformly in $t$:
$$
\sup_{t\in[0,1]} E\Big(\big[x_t - \sum_{k=1}^{\infty}Z_ke_k(t)\big]^2\Big) \mathop{\longrightarrow}_{n\to\infty}0.
$$
Finally, $x_t$ is a Gaussian process if and only if $\{{{Z}}_n\}_{n\geq 1}$ is a sequence of independent Gaussian random variables.
\end{theorem}
\subsection{Representation of the set $\H$ with respect to the Fourier basis $\{\st{k}{t}, \ct{k}{t}\}_{k\geq 0}$}
We deal in this paper with the following class $\H$ of processes, thought of as the 
set of `pure Gaussian noises' on the circumference.
\begin{definition}
Let $ \{x_t\}_{t\in[0,1]}$ be a stochastic process with covariance function
$
C(s,t)=Cov(x_t,x_s) 
$.
$ \H$ is the set of real \emph{Gaussian} stochastic processes 
$ \{x_t\}_{t\in[0,1]}$ such that
\begin{description}
\item[zero-mean:] $E(x_t) = 0$, $\forall t\in[0,1]$;
\item[continuously stationary:]
there exists a continuous real function $\tilde{C}$ such that
$ C(s,t) = \tilde{C}(s-t)  $, $\forall s,t\in[0,1]$;
\item[periodic:] $ \{x_t\}_{t\in[0,1]}$ admits 
a periodic extension to $\R$ (i.e. $ x_{0} = x_{1} , a.s.$).
\end{description}
\end{definition}
\begin{remark}\label{rem:1.1}
A necessary and sufficient condition for a continuously stationary process to be
periodic is that $\tilde{C}(1) =\tilde{C}(0) $. This allows a continuous version of the process
with $ Var(x_{t+1}-x_t)=0$ for any $t\in\R$.
 We remark that if  $ \{x_t\}_{t\in[0,1]}\in\H $  and if $ \tilde{C}(s-t) = C(s,t)  $ is its covariogram function, then $ \tilde{C}(t)  =  \tilde{C}(t+1)$.
\end{remark}

The set $ \H$  is a Hilbert space, when it is equipped with the inner product given by
$
\langle x(\cdot),y(\cdot)\rangle = \int_{0}^{1}{E(x_ty_t)dt} .
$
Karhunen-Lo\`eve's decomposition theorem can be specialized to $\H$, in order to show that a process is in $\H$ if and only if it can be written as 
limit of a canonical trigonometric random series, namely the constant function equal to $1$ together with 
the sequence $\{\st{k}{t},\ct{k}{t}\}_{k\geq 1}$,
where $\st{k}{t}=\d\sin(2k\pi t)$ and $\ct{k}{t}=\d\cos(2k\pi t)$.\\
\begin{theorem}\label{2.1}
 Let $ \{x_t\}_{t\in[0,1]} \in \H $ with covariance $C(s,t) = \tilde{C}(t-s)$; then in mean square, uniformly in $t$, 
$$
x_t = c_0 Y'_0 + \sum_{k=1}^{\infty}c_k
(Y_k\st{k}{t} + Y'_k\ct{k}{t})
$$ 
where
$
\{Y_k\}_{k \geq 1}$ and $\{Y'_k\}_{k \geq 0}$ are two independent sequences of independent standard Gaussian variables,
and
$
\{c_k\}_{k\geq 0}\in \ell^2
$
is such that
$$
c_n^2 = \int_{0}^{1}{\tilde{C}(s)\cos(2n\pi s)ds},\qquad n=0,1,2,\ldots
$$
\end{theorem}
\begin{proof}
See Appendix~\ref{appA}. 
\qed\end{proof}
\begin{theorem}\label{teo:2.2}
Let $
\{Y_k\}_{k \geq 1}$ and $\{Y'_k\}_{k \geq 0}$  be two independent sequences of independent standard Gaussian variables,
and $\{c_k\}_{k\geq 0}\in \ell^2 $.
Then the sequence
$$
y^{(n)}_t = c_0Y'_0+\sum_{k=1}^{n}c_k(Y_k
\st{k}{t}+ Y'_k 
\ct{k}{t})
$$ 
converges in mean square, uniformly in $t$ to $\{y_t\}_{t\in[0,1]}\in\H$.
Moreover if $C(s,t)$ is the covariance function of $y_t$ , 
then uniformly, absolutely and in $ L^2[0,1]\times [0,1] $,
\begin{equation}
\begin{aligned}
C(s,t) & = c_0^2+\sum_{k=1}^{\infty}{c_k^2\ct{k}{s}\ct{k}{t}}
+ \sum_{k=1}^{\infty}{c_k^2\st{k}{s}\st{k}{t}}\\
& = c_0^2+ 2 \sum_{k=1}^{\infty}{c_k^2\cos(2k\pi s)\cos(2k\pi t)}
+ 2 \sum_{k=1}^{\infty}{c_k^2\sin(2k\pi s)\sin(2k\pi t)} \\
& =
c_0^2+2 \sum_{k=1}^{\infty}{c_k^2\cos(2k\pi(s- t))}.
\end{aligned}\label{Cov_X}
\end{equation}
\end{theorem}
\begin{proof}
See Appendix~\ref{appA}.
\qed\end{proof}
\begin{remark}\label{oss:per_proc}
As a consequence of Theorem~\ref{2.1} and Theorem~\ref{teo:2.2}, 
we can observe that periodic processes with period $\frac{1}{m}$ 
have only terms of form $mk$ in their expansion:
$$
x_{t+\frac{1}{m}} = c_0 Y'_0 + \sum_{k=1}^{\infty}c_{mk}
(Y_{mk}\st{mk}{t+\tfrac{1}{m}} + Y'_{mk}\ct{mk}{t+\tfrac{1}{m}})= x_t.
$$ 
More fancy processes having only odd terms are antiperiodic 
with period $\frac{1}{2}$, i.e.
$$
x_{t+\frac{1}{2}} = \sum_{k=0}^{\infty}c_{2k+1}
(Y_{2k+1}\st{2k+1}{t+\tfrac{1}{2}} + Y'_{2k+1}\ct{2k+1}{t+\tfrac{1}{2}})= 
-x_t.
$$
\end{remark}

An immediate consequence of this remark is that when one needs to model a pure noise on the boundary of a circle, 
then he must choose processes whose expansion has both odd and even terms.
\subsection{The quotient set $\H_Z$}\label{sec:H_Z}
It is easy to see that $\H$ can be seen as a Hilbert space, isometrically equivalent to the 
space of the coefficients $ \ell^2 $; 
 let us consider two independent sequences 
$\{{{\bar{Y}}}_n\}_{n\geq 1}$ and $\{{{\bar{Y}'}}_n\}_{n\geq 0} $ 
of independent standard Gaussian variables. For each $  \{z_t\}_{t\in[0,1]}\in \H $, 
there exists an $  \{x_t\}_{t\in[0,1]}\in \H_{Z}  $ having the same law, where
\[
\H_{Z} = \Big\{ \{x_t\}_{t\in[0,1]}\in \H : \\
x_t =a_0\bar{Y'}_0+
\sum_{k=1}^{\infty}a_k(\bar{Y}_k\st{k}{t} + \bar{Y'}_k\ct{k}{t})
, \{a_n\}_{n\geq 0}\in \ell^2  \Big\}
\]
and the limit is in mean square and uniformly in $t$.
From Theorem~\ref{2.1} and Theorem~\ref{teo:2.2} it 
is naturally defined an isometry between the representative space $\H_Z$ and $\ell^2$:
$$
x_t =a_0\bar{Y'}_0 + 
\sum_{k=1}^{\infty}a_k(\bar{Y}_k\st{k}{t} + \bar{Y'}_k
\ct{k}{t}) \longleftrightarrow \{a_0, \sqrt{2}a_1, \sqrt{2}a_2, \sqrt{2}a_3,\ldots\}\in \ell^2,
$$
where $\|x_t\|_{\H_{Z}} = \sqrt{a_0^2+ 2
\sum_n a_n^2}$.

\subsection{The space $\H_0$ and its relation with $\H_Z$}
By Theorem~\ref{teo:2.2} given $\{c_i\}_{i\geq 0} \in \ell^2$, there exists a unique $\{x_t\}_{t\in[0,1]}\in\H_Z$, with
covariance function given by
\[
C(s,t) = c_0^2+\sum_{k=1}^{\infty}{c_k^2\ct{k}{s}\ct{k}{t}}.
\]

Let us define the set $\H_0$ of the process generated by those in $\H$ conditioned to be $0$ at $t=0$.
\begin{definition}
Let $\H_0$ be the following set
\begin{multline*}
\H_0 = \{ \{y_t\}_{t\in[0,1]}\colon \exists \{x_t\}_{t\in[0,1]}\in\H \text{ such that }\\
\L((y_{t_1},\ldots,y_{t_n})) = \L((x_{t_1},\ldots,x_{t_n})|x_0 = 0), \quad 
\forall \underline{t}\in[0,1]^n,  n\in\N\}.
\end{multline*} 
We call:
\begin{description}
\item[Generator process:]  the process $ \{x_t\}_{t\in[0,1]}\in\H$;
\item[Generated process:] the process $ \{y_t\}_{t\in[0,1]}\in\H_0$.
 \end{description}
\end{definition}

In other words, the process $\{x_t\}_{t\in[0,1]}\in \H_Z$, conditioned to be $0$ at $t=0$, is the periodic zero-mean
Gaussian process $\{y_t\}_{t\in[0,1]}\in \H_0$ with covariance
function
\begin{equation}\label{eq:RfromC}
R(s,t) = C(s,t) - \frac{C(s,0) C(0,t) }{C(0,0)}.
\end{equation}

It is easy to show that $  \{y_t\}_{t\in[0,1]} \notin \H  $ because it is not stationary.
However, the function \(R(s,t)\) is symmetric, and hence it is the $L^2$-limit of its 2-D Fourier series.
With the notation given above, with $\ct{0}{t}=1$, we get the series expansion:
\begin{equation}\label{eq:Rexp2a}
R(s,t) 
= 
 \sum_{k,j=0}^{\infty}{r^{cc}_{kj}\ct{k}{s}\ct{j}{t}}
+ \sum_{k,j=1}^{\infty}{r^{ss}_{kj}\st{k}{s}\st{j}{t}}
+ \sum_{k=1,j=0}^{\infty}{r^{sc}_{kj}\st{k}{s}\ct{j}{t}}
+ \sum_{k=0,j=1}^{\infty}{r^{cs}_{kj}\ct{k}{s}\st{j}{t}}.
\end{equation}

The following theorem gives a necessary and sufficient condition for a process
$\{y_t\}_{t\in[0,1]}$ with
covariance function
$R(s,t)$ to have a unique process $\{x_t\}_{t\in[0,1]}\in\H_Z$ which generates it. The trivial
case when $R(s,t)=0$ (generated by a constant process) is omitted since 
it is the sole case when the solution is not unique. The proof may be found
in Appendix~\ref{appA}.
\begin{theorem}\label{thm:RtoS}
For any Gaussian process $\{y_t\}_{t\in[0,1]}$ such that
$y_0 = 0$, $E(y_t)=0$ and continuous
covariance function $R(s,t)\neq 0$, there exists a unique (in law) stationary 
process $\{x_t\}_{t\in[0,1]}\in\H_Z$ which generates $\{y_t\}_{t\in[0,1]}$ if and only if the Fourier
coefficients of $R(s,t)$ satisfy:
\begin{itemize}
\item the mixed matrices $\cos-\sin$  and $\sin-\cos$ are null: 
\[
\{r^{cs}_{jk}\}_{j\geq 0,k\geq 1} = \{r^{sc}_{jk}\}_{j\geq 1,k\geq 0} = 0;
\]
\item the $\sin-\sin$ matrix is a non-negative diagonal in $\ell^1$:
\[
\{r^{ss}_{jk}\}_{j,k\geq 1} = 
\begin{pmatrix}
r^{ss}_{11} & 0 & 0 & 0 & \cdots \\
0 & r^{ss}_{22} & 0 & 0 & \cdots \\
0 & 0 & r^{ss}_{33} & 0 & \cdots \\
\hdotsfor{5} 
\end{pmatrix},
\]
with $r^{ss}_{kk}\geq 0$ and $r^{cc}_{00}<\bar{r} =\sum_k r^{ss}_{kk}<\infty$; 
\item defined $r^{ss}_{00}=\frac{r^{cc}_{00}\bar{r}}{\bar{r}-r^{cc}_{00}}$, 
the $\cos-\cos$ matrix is built from the $\sin-\sin$ matrix and $r^{cc}_{00}$:
\[
\{r^{cc}_{jk}\}_{j,k\geq 0} = 
\begin{pmatrix}
r^{ss}_{00} & 0 & 0 & 0 & \cdots \\
0 & r^{ss}_{11} & 0 & 0 & \cdots \\
0 & 0 & r^{ss}_{22} & 0 & \cdots \\
\hdotsfor{5} 
\end{pmatrix}-
\frac{\bar{r}-r^{cc}_{00}}{\bar{r}^2}
\begin{pmatrix}
r^{ss}_{00}r^{ss}_{00} & r^{ss}_{00}r^{ss}_{11} & r^{ss}_{00}r^{ss}_{22} & \cdots\\
r^{ss}_{11}r^{ss}_{00} & r^{ss}_{11}r^{ss}_{11} & r^{ss}_{11}r^{ss}_{22} & \cdots\\
r^{ss}_{22}r^{ss}_{00} & r^{ss}_{22}r^{ss}_{11} & r^{ss}_{22}r^{ss}_{22} & \cdots\\
\hdotsfor{4} 
\end{pmatrix}.
\]
\end{itemize}
\end{theorem}

\begin{remark}
One of the models mainly used for periodic noise is the Brownian bridge, i.e.\ 
the process  $ \{B_t\}_{t\in[0,1]} $ such that 
$
B_t = W_t - tW_1,
$
where $ \{W_t\}_{t\in[0,1]} $ is a Brownian motion. 
This  process is Gaussian, periodic and has the following standard representation in Fourier random series:
$$
B_t = \sum_{k=1}^\infty Z_k \frac{\sqrt{2} \sin(k \pi t)}{k \pi},
$$
where $Z_1, Z_2, \ldots$ are independent identically distributed standard normal random variables. Starting from the results explained in this section, it is
straightforward to prove that $\{B_t\}_{t\in[0,1]} \notin \H_0$; so we cannot consider it a ``good'' noise model on the boundary of a circle.
\end{remark}

\section{A process in $\H_0$ shares the same asymptotic behavior for the spectrum with its generator}\label{sec:4s}

We want to get information about Fourier coefficients of Karhunen-Lo\`eve expansion  for processes in 
$\H_0$ with respect to the coefficients of their generators in $\H$. To do this, as described in the 
Theorem~\ref{teo:KL},  it is sufficient to study the spectrum of the integral operator induced by the covariance function of the process $  \{y_t\}_{t\in[0,1]}\in \H_0$ 
generated by $ \{x_t\}_{t\in[0,1]}\in\H$.

\begin{theorem}\label{teo:Asympthotic}
Denote by
$ \{y_t\}_{t\in[0,1]}$ a process in $\H_0 $ and by $\{x_t\}_{t\in[0,1]}$ its generator in $\H $. 
Let $ \{c_n\}_{n\geq 0}  \in \ell^2$ be the sequence of Fourier coefficients of Karhunen-Lo\`eve expansion 
of the process $ \{x_t\}_{t\in[0,1]}$, such that, as in Theorem~\ref{2.1} and Theorem~\ref{teo:2.2},
$$
x_t = c_0 Y'_0 + \sum_{k=1}^{\infty}c_k
(Y_k\st{k}{t} + Y'_k\ct{k}{t}). 
$$
Then the Karhunen-Lo\`eve expansion of the process $ \{y_t\}_{t\in[0,1]}$ has the following form:
$$
y_t =   \sum_{k=0}^{\infty}(c_kY_k\st{k}{t} +   \tilde{c}_k  Y'_k f_{k}(t))
$$
where $f_{k}(t)$ is the eigenfunction related to the eigenvalue $\tilde{a}_n=\tilde{c}_n^2$, and, for all $n\in\N$,
\begin{align*}
a_{{k_n}} & = \tilde{a}_{k_n} = a_{k_{n+1}} && \text{if } a_{k_n}=a_{k_{n+1}} \\
a_{{k_n}} & > \tilde{a}_{k_n} > a_{k_{n+1}} && \text{if } a_{k_n}=a_{k_{n+1}} 
\end{align*}
where $\{a_{k_n}\}_{n\geq 0}$ is a decreasing reordering of the sequence $\{a_{n}\}_{n\geq 0}$.
\end{theorem}
\begin{proof}
See Appendix~\ref{appB1}.
\qed\end{proof}
\begin{remark}
Theorem~\ref{thm:RtoS} and Theorem~\ref{teo:Asympthotic} give a theoretical approach to build the Karhunen-Lo\`eve expansion 
of processes in $\H_0$. 
A numerical example of such a procedure may be fond in the example of the Section~\ref{sec:6s}.
\end{remark}
\section{ A process in $\H_0$ shares the same path regularity properties with its generator}\label{sec:5s}
We showed in Theorem~\ref{teo:Asympthotic} that a process in $\H_0$ and its generator in $\H$
share the same asymptotic behavior for the spectrum. 
In this section, we show that the regularity of the paths is also maintained.
\subsection{H\"{o}lder regularity of the paths of processes in $\H$ and in $\H_0$}
We first remind that the H\"{o}lder space $C^{m, \alpha} ([0,1])$, where  $m\geq 0$ is an integer
and $0 < \alpha \leq 1$, consists of those functions on $[0,1]$ having continuous derivatives 
up to order $m$ and such that the $m^{\text{th}}$-derivative is H\"{o}lder continuous with exponent $\alpha$. 
We recall a classic regularity theorem.
\begin{theorem}[Kolmogorov-Centsov continuity criterion, \cite{Revuz}]\label{teo:2.3}
Let $ \{x_t\}_{t\in[0,1]} $ be a real stochastic process such that there exist three positive 
constants $ \gamma $, $c$ and $ \epsilon $ so that 
$$
E\big(|x_t-x_s|^{\gamma}\big) \leq c|t-s|^{1+\epsilon};
$$
then there exists a modification $ \{\tilde{x}_{t}\}_{t\in[0,1]}  $ of $\{x_t\}_{t\in[0,1]}$, such that
$$
E\Big(\Big(\sup_{s\neq t}{\frac{|\tilde{x}_t-\tilde{x}_s|}{|t-s|^{\alpha}}}
\Big)^{\gamma} \Big)<\infty
$$
for all $ \alpha\in [0,\frac{\epsilon}{\gamma}) $; in particular the trajectories of $ \{\tilde{x}_{t}\}_{t\in[0,1]}  $
belongs to $ C^{0, \alpha} ([0,1]) $.
\end{theorem}
 
The following results are an immediate consequence of this last theorem  
(proofs may be found in Appendix~\ref{appC}), where the processes $ \{\tilde{x}_t\}_{t\in [0,1]} $ and
$ \{\tilde{y}_t\}_{t\in [0,1]} $ are thought modified as in the Theorem~\ref{teo:2.3}.
\begin{theorem}\label{teo:2.4}
 Let  $ \{x_t\}_{t\in [0,1]} $, a stationary stochastic process with null expectation, and 
let $R(s,t)$ be its covariance function; 
 if $ R  \in C^{0, \alpha} ([0,1]\times [0,1])$, with $0 < \alpha \leq 1$,
 then almost all trajectories of 
$ \{x_t\}_{t\in[0,1]} $ belong to $  C^{0, \beta} ([0,1])  $ with $ \beta < \frac{\alpha}{2} $.
\end{theorem}

It is simple to apply this last theorem to processes laying in  $ \H $ and in $\H_0$: assume that $ \{x_t\}_{t\in [0,1]}\in\H $ and 
let $C(s,t) = \tilde{C}(s-t)$ be its covariance function. If 
$\tilde{C} \in C^{0, \alpha} ([0,1])$, then almost all trajectories of  $ \{\tilde{x}_t\}_{t\in[0,1]} $ 
belong to $ C^{0, \beta} ([0,1]) $, for any $ \beta < \frac{\alpha}{2} $. The same argument can be applied to $\H_0$ processes.
 
In fact we can say something more.
\begin{theorem}\label{teo:RAndC}
Let $ \{x_t\}_{t\in [0,1]}\in\H $ and 
let $C(s,t) = \tilde{C}(s-t)$ be its covariance function.  Consider 
its generated process $ \{y_t\}_{t\in [0,1]}\in\H_0 $, and let  
$
R(s,t)
$
be its covariance function.
Then we have, for any $ \beta < \frac{\alpha}{2} $,
\[
\tilde{C}  \in C^{0, \alpha} ( [0,1])\Rightarrow R  \in C^{0, \alpha} ([0,1]\times [0,1]) \Rightarrow  
\left\{
\begin{matrix}
\{\tilde{y}_t\}_{t\in[0,1]} \in C^{0, \beta} ([0,1]) \\
\{\tilde{x}_t\}_{t\in[0,1]} \in C^{0, \beta} ([0,1])  
\end{matrix}
\right.
\]
\end{theorem}
 
This last result implies that regularity properties of almost all trajectories of  
$\{\tilde{x}_t\}_{t\in[0,1]}   $ and of its generated process
$ \{\tilde{y}_t\}_{t\in[0,1]} $ have the same lower bound, 
obtained by studying regularity of their covariance function.

\subsection{Upper order regularity }
In Section~\ref{sec:H_Z}, a sequence in $\ell^2$ is uniquely associated 
to each stochastic process in $\H$. We are now showing how the decrease rate of such sequence 
ia associated with the regularity of the process trajectory path.

A very useful result for our analysis will be the following one, whose proof may be found in
 \cite{Boas}.
\begin{theorem}[Boas' Theorem]
Let $f\in L^1[0,1]$ be a function whose Fourier expansion has only nonnegative cosine terms, 
and let $ \{a_n\}_{n\geq 0} $ be the sequence of its cosine coefficient. Then
$$
f\in C^{0, \alpha} ([0,1]) \Longleftrightarrow  a_k  = 
O\Big(\frac{1}{k^{\alpha+1}}\Big).
$$
\end{theorem}

Boas' Theorem may be used in connection with 
Theorem~\ref{2.1} and Theorem~\ref{teo:2.2} to deduce 
more regularity properties of the processes in $ \H $, since $\tilde{C}$ is
a function whose Fourier expansion has only nonnegative cosine terms.
In fact, take $\{c_n\}_{n\geq 0}$ as in Theorem~\ref{2.1} and Theorem~\ref{teo:2.2}.
From Boas' Theorem we have that 
if
$
k^2c_k^2 = O(\frac{1}{k^{1+\alpha}})
$
for $ 0<\alpha\leq 1 $, then $ \tilde{C} \in C^{2, \alpha} ([0,1]) $ . 
This link between the regularity of $\tilde{C}$ and the paths of $ \{x_t\}_{t\in[0,1]} $
is underlined in the following theorem. The proof is in the Appendix~\ref{appC}.
\begin{theorem}\label{teo:5.4}
With the notations of Theorem~\ref{teo:2.2},
if $c_k^2 = O(\frac{1}{k^{3+\alpha}})$, then there exists a version of $ \{x_t\}_{t\in[0,1]} $ whose trajectories belong to $ C^{1, \beta} ([0,1]) $, with $ \beta < \frac{\alpha}{2} $.
\end{theorem}

A natural generalization of this result is the following corollary, exemplified in  Figure~\ref{fig:1}.
\begin{corollary}\label{cor:1}
With the notations of Theorem~\ref{teo:2.2},
if there exist an $m\in\N$ such that $c_k^2 = O({1}/{k^{1+2m+\alpha}})$ 
 then there exists a version of  $ \{x_t\}_{t\in[0,1]} $ whose trajectories belongs to $ C^{m, \beta} ([0,1]) $, with $ \beta < \frac{\alpha}{2} $.
\end{corollary}

\begin{figure}
\begin{center}
\includegraphics[width=0.48\textwidth]{new_fig1.png}
\includegraphics[width=0.48\textwidth]{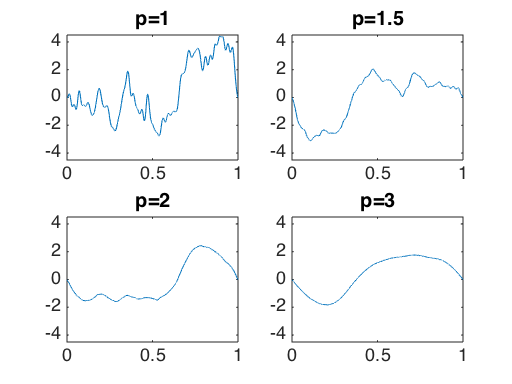}
\caption{Change of path regularity:
a comparison between trajectories of processes for fixed i.i.d.\ gaussian $\{Y_k,Y_k'\}_{k\geq 0}$ 
and varying the magnitude of $\{c_k\}_{k\geq 0}$.
Left: parametric model in $\H$ given in \eqref{model}, where $a=1.5$ (dilation coefficient, fixed)
and varying $p$.
Right: parametric model in $\H_0$ given in \eqref{modelH0}, obtained by conditioning the model \eqref{model} 
to be $0$ at $t=0$. The coefficients $\{\tilde{c}_k\}_{k\geq 0} $ 
and the eigenbase $ \{ f_{k} \}_{k\geq 0}$ of the model 
are obtained as explained in Section~\ref{sec:6s}. 
As a consequence of the Corollary~\ref{cor:1}, each trajectory belongs to $ C^{m, \beta} ([0,1]) $,
where $m+ \beta< 2p - 1$. The series are truncated at $N=40$.}\label{fig:1}
\end{center}
\end{figure}

\section{A parametric model in $\H$ and in $\H_0$}\label{sec:6s}

Results provided in this paper suggests to create a Gaussian parametric family of stationary
and periodic processes of arbitrary 
regularity. In fact, let us consider the following family of processes in $\H$:
\begin{equation}\label{model}
x_t = \sum_{k=1}^{N}\frac{a}{k^{p}}(Y_k\sin(2k\pi t) + Y'_k\cos(2k\pi t)).
\end{equation} 

This family is the discrete approximation of the model given in \cite{HoboltShape},
obtained when $N$ goes to infinite.
We note that for this limiting process,
Theorem~\ref{teo:2.4} states that
the paths become more regular as $p$ increases. 
This property is shown in Figure~\ref{fig:1} (left), which 
suggests how to smooth a process by changing $p$. 

By Theorem~\ref{thm:RtoS}, it is possible to build a parametric model in $\H_0$ of the form given in the
Theorem~\ref{teo:Asympthotic}
\begin{equation}\label{modelH0}
x_t = \sum_{k=1}^{N}\frac{a}{k^{p}}Y_k\sin(2k\pi t) +\sum_{k=0}^{N}  \tilde{c}_k Y'_k f_k(t).
\end{equation} 
The functions $\{f_k(t)\}_{k\geq 0}$ are the eigenfunctions of the $\cos-\cos$ part of the covariance function $R(s,t)$ given in
\eqref{eq:Rexp2a}. To find an approximation of these first eigenfunctions, given 
\begin{equation*}
R^{cc}(s,t) = 
 \sum_{k,j=0}^{N}{R^{cc}_{kj}\ct{k}{s}\ct{j}{t}},
\end{equation*}
we may find the spectral representation of the $\cos-\cos$ matrix $R^{cc} = U D U^T$, with $D$ diagonal and $U$ unitary.
Then $ \tilde{c}_k = \sqrt{D_{kk}}$ and $f_k(t) = \sum_{j=0}^N U_{jk}\ct{j}{t}$.

Model \eqref{model} gives a family of Gaussian processes. 
In application, maximum likelihood estimates of $a$ and $p$ is a straightforward consequence 
of a fast Fourier transform of the
observed discretized process $\{x_t\}_{t\in[0,i/n]}$, $i=0,\ldots,n$. 
The properties of these estimators are studied in the following section.


\subsection{Maximum likelihood estimators of {\eqref{model}}}\label{sec:ML_estim}

Given $(x_{t_0},x_{t_1},\ldots,x_{t_{n}})$ sampled from \eqref{model}, 
we want to find the property of the maximum likelihood estimator 
$(\hat{a},\hat{p})$ of the parameters $(a,p)$.

More precisely, with a equispaced or nonequispaced Fourier transform (see, e.g.,  \cite{FFT:book,FFT_nonUniform}), 
we first transform $(x_{t_0},x_{t_1},\ldots,x_{t_{n}})$ into $(y_1^{(1)},y_2^{(1)},\ldots,y_n^{(1)})$ and
$(y_1^{(2)},y_2^{(2)},\ldots,y_n^{(2)})$ (real and imaginary part). 
As a consequence of Theorem~\ref{2.1} applied to \eqref{model}, there exist 
two sequences $
\{Y_k\}_{k \geq 1}$ and $\{Y'_k\}_{k \geq 1}$  of independent Gaussian standard
random variables such that 
\[
\begin{matrix}
y_1^{(1)}=aY_1,& y_2^{(1)}=\frac{a}{2^p}Y_2, &
\ldots, &y_n^{(1)}=\frac{a}{n^p}Y_n, 
\\
y_1^{(2)}= aY'_1, & y_2^{(2)}=\frac{a}{2^p}Y'_2, &
\ldots, &y_n^{(2)}=\frac{a}{n^p}Y'_n.
\end{matrix}
\]
The log-likelihood function then reads
\begin{align}
\ell_n(a,p) 
& 
= \sum_{k=1}^n \log\Bigg(  \tfrac{1}{\sqrt{2\pi \tfrac{a^2}{k^{2p}}}} \exp\Big({-\tfrac{1}{2} \frac{(y_k^{(1)})^2 }{ \tfrac{a^2}{k^{2p}} } }\Big) \Bigg)
+ 
\sum_{k=1}^n \log\Bigg(  \tfrac{1}{\sqrt{2\pi \tfrac{a^2}{k^{2p}}}} \exp\Big({-\tfrac{1}{2} \frac{(y_k^{(2)})^2 }{ \tfrac{a^2}{k^{2p}} } } \Big)\Bigg) \notag 
\\
& 
= -n\log(2\pi) -2n\log(a) +2p\sum_{k=1}^n\log(k) \notag 
-\frac{1}{2a^2}
\sum_{k=1}^n k^{2p} \big((y_k^{(1)})^2+(y_k^{(2)})^2\big) \notag 
\intertext{and hence, if $o_k= (y_k^{(1)})^2+(y_k^{(2)})^2$, $k=1,\ldots,n$, we get}
\notag
\frac{\partial \ell_n}{\partial a} & = -\frac{2n}{a} + \frac{1}{a^{3}}  
\sum_{k=1}^n k^{2p} o_k 
\\
\frac{\partial \ell_n}{\partial p} 
& 
= 
2\sum_{k=1}^n\log(k) 
-\frac{1}{a^2}
\sum_{k=1}^n \log(k) k^{2p} o_k =
\sum_{k=1}^n 
\log(k)\Big( 2
-\frac{k^{2p} o_k }{a^2}
\Big) \label{eq:MLE_p}
\end{align}
As expected, when $p_0$ is a known parameter,
\[
\hat{a}^2 = \frac{1}{2n} 
\sum_{k=1}^n k^{2p_0} o_k, \qquad 
2n \tfrac{\hat{a}^2}{a_0^2} \sim \chi^2_{2n},
\]
where $\chi^2_{2n}$ is a chi-square distribution with $2n$ degree of freedom,
while nothing is known about the distribution of $\hat{p}$, for small $n$, and
for the distribution of  the couple $(\hat{a},\hat{p})$.
We have the following asymptotic results, whose proof may be found in Appendix~\ref{appD}.
\begin{theorem}\label{model:asympt}
There exists an ML estimator $\{\hat{p}_n\}_{n\geq 1}$, zero of the equation \eqref{eq:MLE_p},
such that
\[
\hat{p}_n \mathop{\longrightarrow}\limits^{a.s.}_{n\to\infty} p_0, \qquad
\frac{\hat{p}_n-p_0}{2 \sqrt{\sum_1^n \log^2(k)}} \mathop{\longrightarrow}\limits^{L}_{n\to\infty}
N(0,1).
\]
Moreover,
\begin{equation}\label{eq:Zcorr}
2\begin{pmatrix}
\frac{\sqrt{n}}{a_0} 
& -\frac{\sum_{k=1}^n 
\log(k)}{\sqrt{n}} 
\\
-\frac{\sum_{k=1}^n 
\log(k)}{a_0\sqrt{\sum_{k=1}^n 
\log^2(k)}} 
 & 
\sqrt{\sum_{k=1}^n \log^2(k)}
\end{pmatrix}
\begin{pmatrix}
\hat{a}_n-a_0
\\
\hat{p}_n-p_0
\end{pmatrix}
\mathop{\longrightarrow}\limits^{L}_{n\to\infty}
\begin{pmatrix}
1
\\
-1
\end{pmatrix}
Z,
\end{equation}
where $Z$ is a standard Gaussian variable.
\end{theorem}

\begin{figure}[h!]
\begin{center}
\includegraphics[width=\textwidth]{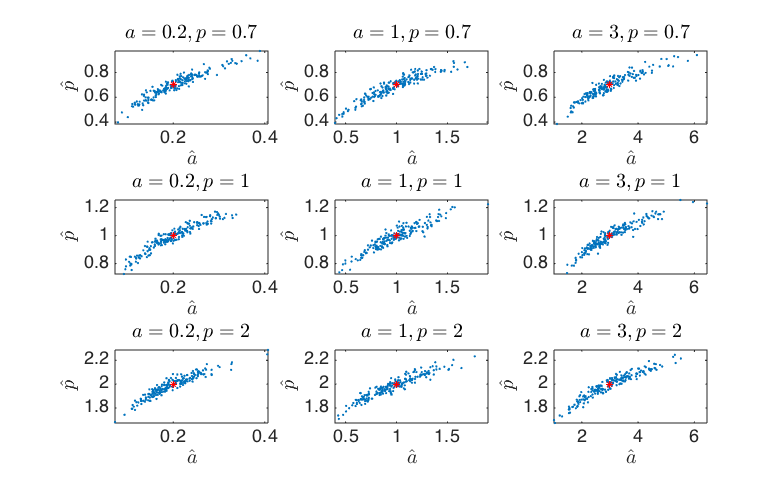}
\caption{Plot of $200$ $(\hat{a},\hat{p})$-joint simulations (blue point) of data coming from
\eqref{model} for different values of $a$ and $p$ (red stars). 
In these pictures, $n=40$ (see Section~\ref{sec:ML_estim} for notations).}\label{fig:2}
\end{center}
\end{figure}

As a corollary of 
Theorem~\ref{model:asympt}, the joint perfect correlation between $\hat{a}_n$
and $\hat{p}_n$ is asymptotically predicted.
In Figure~\ref{fig:2} we show this fact by plotting the maximum likelihood
estimates of $200$ simulated processes from model \eqref{model}, where the 
correlation coefficient $\rho>0.94$ for $n=40$ and different values of $a_0$ and $p_0$. In 
Figure~\ref{fig:3}, we plot the corresponding computed
left-hand part of \eqref{eq:Zcorr}.

\begin{figure}[h!]
\begin{center}
\includegraphics[width=\textwidth]{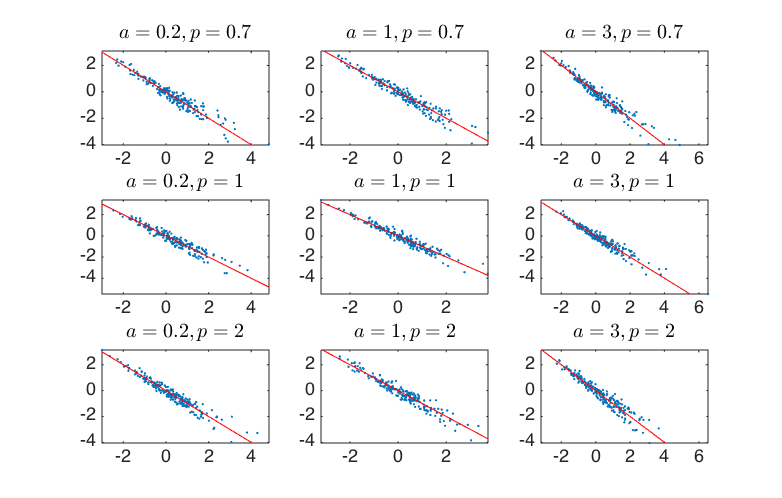}
\caption{Theorem~\ref{model:asympt} predicts that the left-hand part of \eqref{eq:Zcorr} transforms data from Figure~\ref{fig:2} into 
i.i.d.\ vectors with a gaussian distribution concentrated on $y=-x$ (red line).}\label{fig:3}
\end{center}
\end{figure}
\section{Conclusions and perspectives}
The results we presented in this paper, together with the parametric models \eqref{model} and \eqref{modelH0}, are useful to represent the border of circular objects where random noise is present.
The statistical results might help the practitioners to estimate the model parameters; confidence intervals may be found
when dealing with large populations of objects.

A practical advantage of this model is its computer usability: often stochastic processes are represented in computers 
as discrete values, and only a subset of
their Fourier coefficients are available; these models overcome this issue, allowing to find the parameters that best fit the represented process.

In conclusion, the parametric models 
might be considered in the applications more appropriate alternative than the Brownian Bridge. 
Both of the models have a more solid theoretical background,
and they present more flexible in terms of path regularity.

A natural perspective would be to extend the presented results in order to model the shape of the border of a generic $d$-dimensional 
star-shaped object, starting
from the three dimensional case, where the Fourier basis is more treatable, maintaining the usability of the models.

\appendix
\section{Proofs of results of Section~\ref{sec:2s}}\label{appA}

\begin{proof}[Proof of the Theorem~\ref{2.1}]
By Mercer Theorem (see, e.g., \cite{ash}) we know that if $ \{e_n\}_{n\geq 0} $ is an orthonormal basis for the space spanned by the 
eigenfunctions corresponding to nonzero eigenvalues of the integral operator  \eqref{Merc_operator}
then, uniformly, absolutely and in $ L^2[0,1]\times [0,1] $,
\begin{equation}\label{Merc_operator2}
C(s,t) = \sum_{k=0}^{\infty}{e_k(t)e_k(s)\lambda_k},
\end{equation}
where $\lambda_k$ is the eigenvalue corresponding to $e_k$.
By hypothesis, since $C(s,t) = \tilde{C}(|t-s|) = 
\tilde{C}(|t-s|+1)$ by Remark~\ref{rem:1.1}, we get
\begin{equation}\label{eq:CvsS_C}
\int_{0}^{1}{\tilde{C}(s)\cos(2n\pi s)ds} = a_n, \qquad 
\int_{0}^{1}{\tilde{C}(s)\sin(2n\pi s)ds} = 0,
\end{equation}
and hence
\begin{equation}\label{eq:Ctilde}
\tilde{C}(\tau) = a_0 + 2 \sum_{n=0}^\infty a_n \cos(2 n \pi \tau).
\end{equation}
It is simple to prove that the sequence
$ \{\st{n}{t},\ct{n}{t}\}_{n\geq 0} $ contains all the
eigenfunctions of the operator \eqref{Merc_operator}. 
In fact,
\begin{equation}\label{eq:CvsCn}
\begin{aligned}
\int_{0}^{1}{C}(t,\tau)&\ct{n}{\tau} d\tau
=  \d\int_{0}^{1}{\tilde{C}(s)\cos(2n\pi (t+s))dt} 
\\
& =  
\ct{n}{t}\int_{0}^{1}{\cos(2n\pi s)\tilde{C}(s)ds}
- \st{n}{t}\int_{-1/2}^{1/2}{\sin(2n\pi s)\tilde{C}(s)ds} \\
& =  
a_n\ct{n}{t},
\end{aligned}
\end{equation}
the same relation holding when $\ct{n}{t}$ is replaced by $\st{n}{t}$.
By \eqref{eq:Ctilde}, we get
\begin{align*}
C(s,t) & = \tilde{C}(s-t) =
a_0+ \sum_{k=1}^{\infty}{a_k\cos(2k\pi(s- t))}
\\
& = a_0+ 2 \sum_{k=1}^{\infty}{a_k\cos(2k\pi s)\cos(2k\pi t)}
+ 2 \sum_{k=1}^{\infty}{a_k\sin(2k\pi s)\sin(2k\pi t)} \\
& = a_0+\sum_{k=1}^{\infty}{a_k\ct{k}{s}\ct{k}{t}}
+ \sum_{k=1}^{\infty}{a_k\st{k}{s}\st{k}{t}}
\end{align*}
where this equality holds uniformly, absolutely and in $ L^2[0,1]\times [0,1] $
by Mercer Theorem (cfr.\ \eqref{Merc_operator2}).

Now, since $C(s,t)$ is a covariance function, it is positively definite, and hence
$a_n\geq 0$, $\forall n$. Moreover, since $ \{a_n\}_{n\geq 0}\in\ell^1 $,
if we define $c_n = \sqrt{a_n}$, then $ \{c_n\}_{n\geq 0} \in \ell^2$.
From Theorem~\ref{teo:KL} we deduce the existence  
of two independent sequences of independent standard Gaussian variables
$
\{Y_k\}_{k \geq 1}$ and $\{Y'_k\}_{k \geq 0}$ 
such that
in mean square, uniformly in  $ t $
$$
x_t = c_0Y'_0+\sum_{k=1}^{\infty}c_k(Y_k 
\st{k}{t} + Y'_k
\ct{k}{t}). 
$$ 
\qed\end{proof}
\begin{proof}[Proof of the Theorem~\ref{teo:2.2}]
The sequence of Gaussian processes $ y^{(n)}_t $ converges to a periodical 
$ \{y_t\}_{t\in[0,1]} $ in mean square uniformly in $t$, since it is a Cauchy sequence:
$$
\sup_{t\in[0,1]}E[| y^{(n)}_t - y^{(m)}_t |^{2}] =2 
\sum_{k=n}^{m}c_k^2
\mathop{\longrightarrow}_{m,n\to\infty} 0.
$$
Hence, 
$ E[y_t] \equiv 0 $, and 
$$
Cov(y_t,y_s) = c_0^2+
\sum_{k=1}^{\infty}{c_k^2\cos(2k\pi (s-t))}
$$
is a continuous function.
Finally, $  \{y_t\}_{t\in[0,1]} $ is a Gaussian process, since the
two sequences $
\{Y_k\}_{k \geq 1}$ and $\{Y'_k\}_{k \geq 0}$ are formed by independent Gaussian variables.
\qed\end{proof}

\begin{proof}[Proof of the Theorem~\ref{thm:RtoS}]
\emph{Necessity.} Assume there exists a process $\{x_t\}_{t\in[0,1]}\in \H_Z$ which generates $\{y_t\}_{t\in[0,1]}\in \H_0$.
The covariance function $C(s,t)$ of $\{x_t\}_{t\in[0,1]}$ is given as in \eqref{Cov_X}:
\[
C(s,t) = c_0^2+\sum_{k=1}^{\infty}{c_k^2\ct{k}{s}\ct{k}{t}}
+ \sum_{k=1}^{\infty}{c_k^2\st{k}{s}\st{k}{t}}.
\]
If we define 
$x={{C}(0,0)} = \sum_0^\infty c_k^2$, $p_i = c_i^2/x$,  and
\[
D(s,t) = \frac{C(s,t)}{x} = p_0+\sum_{k=1}^{\infty}{p_k\ct{k}{s}\ct{k}{t}}
+ \sum_{k=1}^{\infty}{p_k\st{k}{s}\st{k}{t}},
\]
then, $x>0$ and, by \eqref{eq:RfromC}, we obtain
\begin{equation}\label{eq:Rexp1}
\begin{aligned}
x R(s,t) & =  {D}(s,t) - {D}(0,t) {D}(s,0) \\
&= 
p_0+\sum_{k=1}^{\infty}{p_k\ct{k}{s}\ct{k}{t}}
+ \sum_{k=1}^{\infty}{p_k\st{k}{s}\st{k}{t}} 
- \Big(p_0+\sum_{k=1}^{\infty}p_k\ct{k}{s}
\Big)\Big(
p_0+\sum_{k=1}^{\infty}p_k\ct{k}{t}
\Big).
\end{aligned}
\end{equation}
\eqref{eq:Rexp1} and \eqref{eq:Rexp2a} give
\begin{align}
\label{eq:Rss_fin}
x r^{ss}_{kj} & = 
\begin{cases}
p_k & \text{if }k=j >0\\ 
0 & \text{if }k\neq j 
\end{cases}\\
\label{eq:Rcc_fin}
x r^{cc}_{kj} & = 
\begin{cases}
p_k - p_k^2 & \text{if }k=j \geq 0\\
-p_kp_j & \text{if }k\neq j 
\end{cases}\\
\notag
r^{sc}_{kj} & = r^{cs}_{kj} 
=0
\end{align}
Since $\sum_0^\infty p_k = 1$, if $\bar{r} = \sum_{k=1}^\infty  r^{ss}_{kk}$, we obtain by \eqref{eq:Rss_fin}
\[
x\bar{r} = 1-p_0.
\]
Assume $\bar{r}=0$, then $p_0=1$, which is absurd since $R(s,t)\neq 0$. 
Hence $\bar{r}>0$, and we define
\begin{equation}\label{eq:xdim}
x = \frac{\bar{r}-r^{cc}_{00}}{\bar{r}^2}, 
\end{equation}
Thesis follows by combining \eqref{eq:xdim} and \eqref{eq:Rcc_fin}.

\bigskip

\emph{Sufficiency.} Given the matrices of the 2-D Fourier series as in the theorem assumption,
set $x>0$ as in \eqref{eq:xdim}. Define
\[
p_k = 
r^{ss}_{kk}\frac{\bar{r}-r^{cc}_{00}}{\bar{r}^2}, \qquad
p_0 = \frac{r^{cc}_{00}}{\bar{r}}.
\]
Then $\{p_k\}_{k\geq 0}$ is a non-negative sequence such that $\sum_kp_k=1$. 
Define
\[
x_t = \sqrt{xp_0}Y'_0+\sum_{k=1}^{n}\sqrt{xp_k}(Y_k
\st{k}{t}+ Y'_k 
\ct{k}{t}).
\]
By Theorem~\ref{2.1}, we have
\[
C(s,t) = x\Big(p_0+\sum_{k=1}^{\infty}{p_k\ct{k}{s}\ct{k}{t}}
+ \sum_{k=1}^{\infty}{p_k\st{k}{s}\st{k}{t}}\Big).
\]
It is straightforward to check that \eqref{eq:Rss_fin} and \eqref{eq:Rcc_fin} hold.
The fact that the solution is unique follows immediately from the necessary condition.
\qed\end{proof}

\section{Proof of the Theorem~\ref{teo:Asympthotic}}\label{appB1}
The case $x_t\equiv k$ is obvious. 
Let $ C(t,s) = \tilde{C}(t-s) $  be the covariogram function of $\{x_t\}_{t\in[0,1]}$
(see \eqref{Cov_X} for its expansion). 
Since $x_t\equiv k\iff \tilde{C}(0)=0$, we assume, without loss of generalities,
that $\tilde{C}(0)=1$.

A straightforward computation gives that, if $ \{y_t\}_{t\in[0,1]}\in\H_0 $
is generated by $ \{x_t\}_{t\in[0,1]}\in\H $, 
then  $\{y_t\}_{t\in[0,1]}$ is a Gaussian process with null expectation 
and continuous covariance function
\begin{equation}\label{eq:R}
R(t,s) =  \tilde{C}(t-s) - \frac{\tilde{C}(t)\tilde{C}(s)}{\tilde{C}(0)}
= \tilde{C}(t-s) - {\tilde{C}(t)\tilde{C}(s)}.
\end{equation}

Hence, given the covariogram function $C(s,t)=\tilde{C}(t-s)$ of the generating process $ \{x_t\}_{t\in[0,1]}$, we need to study the spectrum of the operator \eqref{Merc_operator},
where $C$ is replaced by $R$ given in \eqref{eq:R}.

As in \eqref{eq:Ctilde} and \eqref{Cov_X}, we write
$
\tilde{C}(t) = a_0 + 2\sum_{n=1}^{\infty}a_n\cos(2n\pi t)
$ 
with 
$
1 = a_0 + 2\sum_{n=1}^{\infty}a_n
$
since $\tilde{C}(0)=1$.
Let  $f(t)$ be an eigenfunction of \eqref{eq:R}; 
from the expansion theorem (see \cite{ash}) we have in $ L^2[0,1] $,
\begin{equation}\label{eq:expans_f}
f(t) =f_0 +  \sum_{n=1}^{\infty}f_n^{c}\ct{n}{t} + f_n^{s}\st{n}{t}.
\end{equation}
where
$
f_0 = \int_{0}^{1}f(\tau)d\tau,
$
$
f_n^{c} = \int_{0}^{1} \ct{n}{\tau}f(\tau)d\tau 
$
and
$
f_n^{s} = \int_{0}^{1} \st{n}{\tau}f(\tau)d\tau . 
$
Let's look for the eigenvalue related to $f$:
\begin{equation}\label{eq:eigenvalR}
\int_{0}^{1}{R(s,t)f(t)dt} = \int_{0}^{1}{\tilde{C}(t-s)f(t)dt} -\tilde{C}(s) \int_{0}^{1}{\tilde{C}(t)f(t)dt}  = \tilde{a} f(s).
\end{equation}
Substituting \eqref{eq:expans_f} into \eqref{eq:eigenvalR}, and integrating 
with the results in \eqref{eq:CvsS_C} and \eqref{eq:CvsCn}, yields
\begin{equation}\label{eq:num2}
a_0f_0 +\sum_{n=1}^{\infty}a_n(f_n^{c} \ct{n}{s} + f_n^{s} \st{n}{s}) - 
\tilde{C}(s)\Big(a_0 f_0+\d\sum_{n=1}^{\infty}a_n f_n^{c} \Big)  = \tilde{a} f(s).
\end{equation}
\subsection{$\st{n}{s}$ eigenfunctions}
For any $a_n\neq 0$, it is straightforward to see that $f(s)=\st{n}{s}$ is an eigenfunction,
by a direct substitution in \eqref{eq:num2}, and that $\tilde{a}=a_n$.
Moreover, we are going to state more: the only eigenfunctions which contains
some $f_k^{s}\neq 0$ are indeed $\st{n}{s}$ (when $a_n\neq 0$).

Assume that {$\exists k\colon f_k^{s}\neq 0$} and, by contradiction, $f(t)\neq \st{k}{t}$.\\
By multiplying both members of \eqref{eq:num2} by
$ \st{k}{s} $ and integrating, we obtain
$a_kf_k^{s}  = \tilde{a} f_k^{s} $, i.e.,
$
a_k= \tilde{a} .
$
Since $a_k\neq 0$, then $\st{k}{t}$ is an eigenfunction. 
This eigenfunction is orthogonal to
$f(s)$ by Mercer Theorem, and hence
$$
0 = \int_{0}^{1}\st{k}{s}f(s)ds = f_k^{s}.
$$
Summing up, for any $a_n\neq 0$, $\st{n}{t}$ is an eigenfunction associated to
$\tilde{a}=a_n$, and the other eigenfunctions do not contain the terms in $\{\st{n}{t}\}_{n\geq 1}$
(they are even function).

\subsection{The other eigenfunctions of \eqref{eq:num2}.}
To conclude the proof, we should find another sequence of eigenfunctions with 
eigenvalues $\{\tilde{a}_n\}_{n\geq 1}\asymp \{a_n\}_{n\geq 1}$. We will first obtain a simple result on the 
coefficients of the eigenfunctions. Then we will introduce the multiplicity of the
eigenvectors $\{a_n\}_{n\geq 1}$ in order to conclude the proof accordingly.

\medskip

The other eigenfunction takes the form
$
f(t) =f_0 +  \sum_{k=1}^{\infty}f_k^{c}\ct{k}{t}
$.
By multiplying both members of \eqref{eq:num2} by
$ \ct{n}{s} $ and integrating, we obtain
\begin{equation}\label{eq3}
\begin{cases}
a_0 f_0 - a_0(a_0 f_0 +\d\sum_{k=1}^{\infty}a_k f^{c}_k ) =
\tilde{a} f_0,
& n= 0; \\
a_n f^{c}_n - \d a_n(a_0 f_0 +
\d \sum_{k=1}^{\infty}a_k f^{c}_k ) = \tilde{a} f^{c}_n , 
& n>0 .\end{cases} 
\end{equation}
As an immediate consequence, $(a_n = 0) \Rightarrow (f_n^{c}=0)$.
\begin{lemma}\label{lem:sumfc}
$ \{f_n^{c}\}_{n\geq 0} \in \ell^1$, and
$
f_0 + \d\sum_{n=1}^{\infty}f_n^{c}= 0. 
$
\end{lemma}
\begin{proof} 
Recall that $a_n\geq 0$, and that $a_0 + 2\sum_{n=1}^{\infty}a_n = \tilde{C}(0) = 1 $.
For $n>0$, by \eqref{eq3}, we have
$$
| f_n^{c} | 
\leq 
\frac{a_n |f^{c}_n| - \d a_n(a_0 |f_0| +
\d \sum_{k=1}^{\infty}a_k |f^{c}_k| )}{\tilde{a}},
$$
and since $\{a_k |f^{c}_k|\}_{k\geq 0}\in \ell^1$ (as a product of two  $\ell^2$ sequences), 
and $\{a_n\}_{n\geq 1}\in \ell^1$, 
we obtain the first part of the thesis. By \eqref{eq3}
and $a_0 + 2\sum_{n=1}^{\infty}a_n = \tilde{C}(0) =1 $, we get
\[
\begin{aligned}
f_0 + \d\sum_{n=1}^{\infty}f_n^{c} & = 
\frac{a_0 f_0 - a_0(a_0 f_0 +\d\sum_{k=1}^{\infty}a_k f^{c}_k )}{\tilde{a}} 
+ 
\d\sum_{n=1}^{\infty}
\frac{a_n f^{c}_n - \d a_n(a_0 f_0 +
\d \sum_{k=1}^{\infty}a_k f^{c}_k )}{\tilde{a}} 
\\
& = 
\frac{a_0 f_0 +\d\sum_{n=1}^{\infty}a_n f^{c}_n }{\tilde{a}} 
- \frac{a_0 f_0 +\d\sum_{k=1}^{\infty}a_k f^{c}_k }{\tilde{a}}
\Big(   a_0 + 2\sum_{n=1}^{\infty}a_n \Big)
=0.  
\end{aligned}
\]
\qed\end{proof}
\begin{definition}[Multiplicity and support]
Given $\{a_n\}_{n\geq 1}$, we define the \emph{support} $S_{\tilde{a}}$ of $\tilde{a}$:
\[
S_{\tilde{a}} = \{k \colon a_k = \tilde{a}\}.
\]
The \emph{multiplicity} $m_{\tilde{a}} $ of a number $\tilde{a}>0$ is  
the cardinality of $S_{\tilde{a}}$:
\[
m_{\tilde{a}} = \#\{k \colon a_k = \tilde{a}\}.
\]
\end{definition}
It is clear that $m_{\tilde{a}}<\infty$ because $\{a_n\}_{n\geq 1}\in\ell^1$.
\begin{lemma}
If $m_{\tilde{a}} = k>0$, then there are exactly $k-1$ orthogonal eigenfunctions 
of $R$ related to $\tilde{a}$. 
Moreover for anyone of these $k-1$ eigenfunctions, 
\[
n \not\in S_{\tilde{a}} \qquad \Longrightarrow\qquad f_n^{c} = 0 .
\] 
\end{lemma}
\begin{proof}
Let $\tilde{a}>0$ be such that $m_{\tilde{a}}>1$.

It is simple to prove that there always exist
$m_{\tilde{a}}-1$ orthogonal eigenfunctions related to $\tilde{a}$ with
$f_n^{c} = 0$ if $a_n\not\in S_{\tilde{a}}$.
We have two possibilities:
\begin{itemize}
\item $0\in S_{\tilde{a}}$ or, equivalently, $a_0 = \tilde{a}$. 
In this case, \eqref{eq3} is equivalent to the following system
\[
\begin{cases}
f_n^{c}=0, & n \not\in S_{\tilde{a}}\\
\tilde{a} \big(f_0 + \d
\sum_{n\in S_{\tilde{a}}\setminus\{0\}}  f_n^{c} \big)=0.
\end{cases}
\]
\item $0\not\in S_{\tilde{a}}$. 
In this case, \eqref{eq3} is equivalent to the following system
\[
\begin{cases}
f_n^{c}=0, & n \not\in S_{\tilde{a}}\\
\tilde{a} \big(\sum_{n\in S_{\tilde{a}}}  f_n^{c} \big)=0.
\end{cases}
\]
\end{itemize}
In both cases, there exist a $k-1$-dimensional orthogonal basis for the solution
system.

We now need to prove that there are not other eigenfunctions related to $\tilde{a}$. 
Assume that $f_{\bar{n}}^{c}\neq 0$. 
We recall that this fact implies $a_{\bar{n}}\neq 0$.
If $\bar{n}=0$, from \eqref{eq3} we have that 
\[
\begin{cases}
\frac{(a_0 -\tilde{a} )f_0}{a_0}
 = a_0 f_0 +\d\sum_{k=1}^{\infty}a_k f^{c}_k ,
\\
\frac{(a_n -\tilde{a} )f^{c}_n }{\d a_n}
= a_0 f_0 +
\d \sum_{k=1}^{\infty}a_k f^{c}_k  , 
& n \in S_{\tilde{a}}
\end{cases}
\]
The second equation shows that $a_0 f_0 +
\d \sum_{k=1}^{\infty}a_k f^{c}_k=0$, since $a_n = \tilde{a}$, and hence $a_0=\tilde{a}$,
which means that $\bar{n}\in S_{\tilde{a}}$.
Analogously, if $\bar{n}\neq 0$, from \eqref{eq3} we can prove that
$\bar{n}\in S_{\tilde{a}}$, that completes the proof.
\qed\end{proof}

Let $\{a_{(n)}\}_{n\geq 1}$ be the decreasing reordering of the sequence $\{a_n\}_{n\geq 1}$, positive and 
without repetition:
$a_{(1)}>a_{(2)}>\cdots>a_{(n)}>\cdots$ and $\forall a_n>0$, exists $k$ such that 
$a_n =a_{(k)}$.
To conclude the proof, we must find a sequence of eigenvalues $\{\tilde{a}_n\}_{n\geq 1}$ such
that $a_{(n)}>\tilde{a}_n>a_{(n+1)}$. 

\begin{lemma}
For each $n\in\mathbb{N}$, there exists a unique eigenvalue $\tilde{a}_n$ such that
$a_{(n)}>\tilde{a}_n>a_{(n+1)}$. Moreover, $m_{\tilde{a}_n}=1$.
\end{lemma}
\begin{proof}
We have already observed that $a_0 f_0 +\d\sum_{k=1}^{\infty}a_k f^{c}_k=0$
implies, for any $n$, $a_n=\tilde{a}$ or $f_n^{c}=0$. Hence, 
without loss of generalities, we assume $a_0 f_0 +\d\sum_{k=1}^{\infty}a_k f^{c}_k=c
\neq 0$ and we continue the proof. From \eqref{eq3}, we obtain
\begin{equation}\label{eq3b}
f_0 
= c\frac{a_0}{a_0 - \tilde{a}},
\qquad f^{c}_n 
= c\frac{\d a_n }{a_n - \tilde{a}}.
\end{equation}
These relations with, again, $a_0 f_0 +\d\sum_{n=1}^{\infty}a_n f^{c}_n=c$, imply
\begin{equation}\label{eq3c}
\frac{a_0^2}{a_0 - \tilde{a}} + 2\sum_{n=1}^{\infty} \frac{ a_n^2 }{a_n - \tilde{a}}=1.
\end{equation}
We are going to show that there exists a unique solution $\tilde{a}_n$ of \eqref{eq3c}
such that $a_{(n)}>\tilde{a}_n>a_{(n+1)}$. This solution is the searched eigenvalue, 
whose corresponding eigenfunction' expansion is given in \eqref{eq3b}. 

Let us consider the series
\[
S(x) = \frac{a_0^2}{a_0 - x} + 2\sum_{n=1}^{\infty} \frac{ a_n^2 }{a_n - x}
\]
and the derivative series
\[
S'(x) = \frac{a_0^2}{(a_0 - x)^2} + 2\sum_{n=1}^{\infty} \frac{ a_n^2 }{(a_n - x)^2}
\]
then they converge absolutely in each compact set not containing 
$\{a_n\}_{n\geq 1}$.  We have that 
\[
\text{dom}(S) = \text{dom}(s) = \cup_n (a_{(n+1)},a_{(n)}) , 
\qquad S'(x) = s(x),  \forall x \in \text{dom}(S).
\]
Moreover for each $n$,
\[
\lim_{x\to a_{(n+1)}^+} S(x) = -\infty, \qquad \lim_{x\to a_{(n)}^-} S(x) = +\infty, \qquad 
S'(x) > 0 , \forall x\in  (a_{(n+1)},a_{(n)}) .
\]
Hence, there exists a unique $\tilde{a}_n\in (a_{(n+1)},a_{(n)}) $ 
such that $S(\tilde{a}_n)=1$, i.e.\ for which \eqref{eq3c} holds. 
The unique corresponding eigenfunction is given by \eqref{eq3b}, that implies also
$m_{\tilde{a}_n}=1$:
\[
f (t) = 
\frac{a_0}{a_0 - \tilde{a}_n} + \d \sum_{n=1}^\infty
\frac{a_n }{a_n - \tilde{a}_n} \ct{n}{t}.
\]
To complete the proof,
we show that there are not eigenvalues greater then $a_{(1)}=\max_n a_{n}$
or smaller than any $a_{n}>0$.

In fact, if we assume that there exists an eigenvalue $\hat{a} > \max a_n$, then
\eqref{eq3b} shows that the sequence $\{f^{c}_k\}_{k\geq 0}$ is made of either nonnegative or 
nonpositive
numbers, that together with Lemma~\ref{lem:sumfc} implies $f^{c}_k=0$, for any $f$.

In the same way it can be shown that there are no eigenvalues smaller than any $a_n>0$.
\qed\end{proof}

\section{Proofs of results of Section~\ref{sec:5s}}\label{appC}
We simply deduce the results basing on the fact that if $ Y\approx N(0,\sigma^2) $, then 
$
E(|Y|^p) = \sigma^p \frac{2^{\frac{p}{2}}\Gamma\big(\frac{p+1}{2}\big)}{\sqrt{\pi}} ,
$ (see, e.g., \cite{normale}).
\begin{proof}[Proof of the Theorem~\ref{teo:2.4}]
Observe that
$$
E(|x_{t+h} - x_t|^2) = E(x_t^2 + x_{t+h}^2 - 2x_{t+h}x_{t})  =
 R(t+h,t+h) + R(t,t) - 2R(t+h,t) .
$$
Since there exists an $M$ such that
$
|R(s+\delta_1,t+\delta_2)-R(s,t)|\leq M \| (\delta_1,\delta_2) \|^{\alpha}
$,
then there exists a $D$ such that
$$
E(|x_{t+h} - x_t|^2) \leq D|h|^{\alpha}.
$$
The thesis follows.

\qed\end{proof}
\begin{proof}[Proof of the Theorem~\ref{teo:RAndC}]
The first part of the theorem is a simple calculation. The second holds
is a consequence of Theorem~\ref{teo:2.3}, since
$$
\begin{aligned}
E(|x_{t+h} - x_t|^2) 
& = 
E(E(|x_{t+h} - x_t|^2|x_0)) =
 E(E((x_{t+h}- x_0) - (x_t-x_0))^2|x_0)) 
 \\
& = 
R(t+h,t+h) + R(t,t) - 2R(t+h,t)  \leq D|h|^{\alpha}.
\end{aligned}
$$
\qed\end{proof}

\begin{proof}[Proof of the Theorem~\ref{teo:5.4}]
It is clear that 
$$
\partial^2 \tilde{C}(\delta) = 2\partial^2 \sum_{k=1}^{\infty}{c^2_k\cos(2k\pi(\delta))} = -2\sum_{k=1}^{\infty}{(2\pi)^2k^2c^2_k\cos(2k\pi(\delta))}
$$
and that $ \partial^2 \tilde{C} \in C^{0, \alpha} ([0,1])$, for some $ 0<\alpha\leq1 $.
Moreover we have that uniformly in $t$ and in mean square
$$
x_t = c_0Y'_0 + \sum_{k=1}^{\infty}c_k(Y_k\st{k}{t} + Y'_k\ct{k}{t}).
$$
and, from Theorem~\ref{2.1}, there also exist a stochastic process in $ \H $ such that uniformly in $t$ and in mean square
$$
\tilde{x}_t = 2\pi\sum_{k=1}^{\infty}kc_k(Y_k\ct{k}{t} - Y'_k\st{k}{t}),
$$
which has covariogram function belonging to $C^{0, \alpha} ([0,1])$ given by
$$
\tilde{\bar C}(\delta) =  2\sum_{k=1}^{\infty}{(2\pi)^2k^2c^2_k\cos(2k\pi(\delta))}.
$$
If we define
$$
\begin{aligned}
y^{(n)}_t &:= c_0Y'_0 +  \sum_{k=1}^{n}c_k(Y_k\st{k}{t} + Y'_k\ct{k}{t})
\\
\tilde{y}^{(n)}(t) &:= 2 \pi\sum_{k=1}^{n}kc_k(Y_k\ct{k}{t} - Y'_k\st{k}{t}),
\end{aligned}
$$
than $
y^{(n)}_t = {y}^{(n)}_0 + \int_{0}^{t}\tilde{y}^{(n)}_{\tau}d\tau
$,
a.s. for any $n$, while
for each fixed $t$, in mean square we have
$
  \int_{0}^{t}\tilde{y}^{(n)}_{\tau}d\tau \to \int_{0}^{t}\tilde{x}_{\tau}d\tau.
$
Since
\begin{multline*}
\sqrt{E\Big([x_t - x_0 -  \int_{0}^{t}\tilde{x}_{\tau}d\tau]^2\Big)}  
\leq \sqrt{E\Big([x_t - y^{(n)}_t]^2\Big)} 
+ \sqrt{E\Big([y^{(n)}_0 
+ \int_{0}^{t}\tilde{y}^{(n)}_{\tau}d\tau- x_0 - \int_{0}^{t}\tilde{x}_{\tau}d\tau]^2\Big)}
\mathop{\longrightarrow}_{n\to\infty} 0, 
\end{multline*}
it follows that a.s. 
$
x_t = x_0 + \int_{0}^{t}\tilde{x}_{\tau}d\tau .
$
By Theorem~\ref{teo:2.4} we know that almost all trajectory path of $ \tilde{x}_t$  belongs to $C^{0, \beta} ([0,1])$, with   $ \beta < \frac{\alpha}{2} $,
and thesis follows.
\qed\end{proof}

\section{Proofs of results of Section~\ref{sec:6s}}\label{appD}
\begin{proof}[Proof of Theorem~\ref{model:asympt}]
Assume $p_0$ be the true parameter, we may define a sequence
of i.i.d.\ random variables $\{Z_k\}_{k\geq 0}$ in the following way:
\begin{equation}\label{eq:Zk}
Z_k \sim \frac{k^{2p_0} o_k}{a_0^2} \sim \chi^2_2 \sim \exp( \tfrac{1}{2}) .
\end{equation}
Equation~\eqref{eq:MLE_p}, as a function of $(p,p_0)$ and $\{Z_k\}_{k\geq 1}$, becomes
\[
\frac{\partial \ell_n}{\partial p} =
\sum_{k=1}^n 
\log(k)\Big( 2
-k^{2(p-p_0)} Z_k
\Big).
\]
With the notation of \cite[pp.~155-161]{Hall_Heyde}, we have
\begin{align*}
I_n(p) & = \sum_1^n \log^2(k) E\Big( (2
-k^{2(p-p_0)} Z_k)^2| Z_1,\ldots, Z_{k-1} \Big) \\
& =
\sum_1^n \log^2(k) 2( 1 + ( 1 - 2k^{2(p-p_0)})^2 ) ,
\\
J_n(p) & = -\frac{2}{a_0^2}
\sum_{k=1}^n \log^2(k) k^{2p} o_k = 
-2
\sum_{k=1}^n \log^2(k) k^{2(p-p_0)} Z_k
\end{align*}
and, in particular,
\begin{align}
\label{eq:Inp0}
I_n(p_0) & = 4 \sum_1^n \log^2(k) ,
\\
\label{eq:Jnp0}
J_n(p_0) & = -2 \sum_1^n \log^2(k) Z_k .
\end{align}

The thesis is a consequence of \cite[pp.~155-161]{Hall_Heyde},
where the Assumption~1 and Assumption~2 on page 160
guarantee the existence of a ML
estimator $\{\hat{p}_n\}_{n\geq 1}$ such that
\[
\hat{p}_n \mathop{\longrightarrow}\limits^{a.s.}_{n\to\infty} p_0, \qquad
\frac{\hat{p}_n-p_0}{\sqrt{I_n(p_0)}} \mathop{\longrightarrow}\limits^{L}_{n\to\infty}
N(0,1).
\]

\noindent\textbf{Check of \cite[Assumption~1, p.~160]{Hall_Heyde}}.
The fact that $I_n(p_0) \mathop{\longrightarrow}\limits^{a.s.}_{n\to\infty} \infty$ 
is a consequence of \eqref{eq:Inp0}. As $I_n(p_0)=E(I_n(p_0))$, then 
$I_n(p_0)/E(I_n(p_0))\to 1$ uniformly on compacts. By \eqref{eq:Inp0} and \eqref{eq:Jnp0}, 
we have
\[
\frac{J_n(p_0)}{I_n(p_0)} = \frac{-2 \sum_1^n \log^2(k) Z_k}{ 4 \sum_1^n \log^2(k)},
\]
and hence, by \eqref{eq:Zk}, we have
\[
E\Big( \frac{J_n(p_0)}{I_n(p_0)} \Big) = 1, \qquad 
Var\Big( \frac{J_n(p_0)}{I_n(p_0)} \Big) = \frac{\sum_1^n \log^4(k)}{\big(\sum_1^n \log^2(k))^2}
\]
Since, for $n\geq 4$,
\begin{equation}\label{eq:dis_InJn}
\frac{ \log^2(n) }{ \sum_1^n \log^2(m)} 
 \leq
\frac{ 1 }{ \sum_{n/2}^n \big(\frac{\log(n/2)}{\log(n)}\big)^2} 
\leq
\frac{ 1 }{ \sum_{n/2}^n \big(\frac{1}{2}\big)^2} 
\leq
\frac{8}{n} 
\end{equation}
then $\sum_{n=1}^\infty  \big(\frac{\log^2(n)}{{\sum_1^n \log^2(m)}}\big)^{2}<\infty$, 
and hence
\(
Var\big( \frac{J_n(p_0)}{I_n(p_0)} \big) \to 0
\) by Kronecker's Lemma,
which ensures that $I_n(p_0)/E(I_n(p_0))\to -1$ in probability uniformly on compacts.

\noindent\textbf{Check of \cite[Assumption~2, p.~160]{Hall_Heyde}}.
Since, for any $p$, $E_{p}(I_n(p))$ does not change, then Assumption~2.i) is automatically
satisfied.

Now, if $|p_n-p_0|\leq \delta/\sqrt{I_n(p_0)}$, we get
\begin{equation}\label{eq:incr_Jn}
|J_n(p_n) - J_n(p_0) | \leq 
2
\sum_{k=1}^n \log^2(k) \big(k^{\frac{\delta}{\sqrt{\sum_1^n \log^2(m)}}}-1\big) Z_k
\end{equation}
\[
|I_n(p_n) - I_n(p_0) | \leq 
\sum_{k=1}^n 
\log^2(k) 8 k^{2\frac{\delta}{\sqrt{I_n(p_0)}}} (k^{2\frac{\delta}{\sqrt{I_n(p_0)}}}-1)
\]
Note that, since $k\leq n$, we have
\begin{equation*}
1\leq k^{2\frac{\delta}{\sqrt{I_n(p_0)}}} \leq e^{2\frac{\delta}{\sqrt{I_n(p_0)}}\log(n)}
\leq \exp(2\delta)
\end{equation*}
and hence, for sufficient large $n$ and $k\leq n$, 
since 
\begin{equation}\label{diseq:k^delta/In}
k^{2\frac{\delta}{\sqrt{I_n(p_0)}}}-1
\leq C_0 {2\frac{\delta}{\sqrt{I_n(p_0)}}\log(k)},
\end{equation} 
we obtain
\[
\Big|\frac{I_n(p_n) - I_n(p_0)}{I_n(p_0)} \Big| \leq 
\frac{C_1\frac{\sum_{k=1}^n \log^3(k)}{\sqrt{\sum_1^n \log^2(k)}}}{4\sum_1^n \log^2(k)}
= C_2\sum_{k=1}^n  \Big(\frac{\log^2(k)}{{\sum_1^n \log^2(m)}}\Big)^{\frac{3}{2}}
\]
By \eqref{eq:dis_InJn},
then $\sum_{n=1}^\infty  \big(\frac{\log^2(n)}{{\sum_1^n \log^2(m)}}\big)^{\frac{3}{2}}<\infty$, 
and hence, by Kronecker's Lemma, we get Assumption~2.ii), namely
\[
\Big|\frac{I_n(p_n) - I_n(p_0)}{I_n(p_0)} \Big| \to 0.
\]

The last Assumption~2.iii) requires that 
\[
 \frac{J_n(p_n) - J_n(p_0)}{I_n(p_0)} 
\to 0, \qquad \text{a.s.}
\]
To check this, we first note that
\[
\frac{\sum_{k=1}^n \log^2(k) \big(k^{\frac{\delta}{\sqrt{\sum_1^n \log^2(m)}}}-1\big) }
{\sum_1^n \log^2(k)} \to 0,
\]
as a consequence of Kronecker's Lemma, \eqref{diseq:k^delta/In}
and \eqref{eq:dis_InJn}. Then
\[
\Big| E \Big(\frac{J_n(p_n) - J_n(p_0)}{I_n(p_0)} \Big)\Big| 
\leq \frac{E(|J_n(p_n) -J_n(p_0)|)}{I_n(p_0)} \to 0,
\]
and hence, a sufficient condition for
\(
 \frac{J_n(p_n) - J_n(p_0)}{I_n(p_0)}  
\to 0
\) to hold, is that
\begin{equation}\label{eq:varJ_nto0}
Var \Big(\frac{J_n(p_n) - J_n(p_0)}{I_n(p_0)} \Big)
\to 0.
\end{equation}
By \eqref{eq:incr_Jn}, since $Var(X_k) = 4$, we obtain
\[
Var ({J_n(p_n) - J_n(p_0)}) 
\leq 8
\sum_{k=1}^n \log^4(k) \big(k^{\frac{\delta}{\sqrt{\sum_1^n \log^2(m)}}}-1\big)^2.
\]
Again, by \eqref{diseq:k^delta/In}, we obtain
\[
Var \Big(\frac{J_n(p_n) - J_n(p_0)}{I_n(p_0)} \Big)
\leq 
\frac{C_1\frac{\sum_{k=1}^n \log^6(k)}{{\sum_1^n \log^2(k)}}}{\Big(4\sum_1^n \log^2(k)\Big)^2}
= C_2\sum_{k=1}^n  \Big(\frac{\log^2(k)}{{\sum_1^n \log^2(m)}}\Big)^{3}
\]
As above, by \eqref{eq:dis_InJn} and Kronecker's Lemma, we obtain \eqref{eq:varJ_nto0}.

\bigskip 

We sketch the second part of the proof, 
with the notation of \cite[pag.191]{Heyde}.
If we define
\[
\mathbf{G}_n(\boldsymbol{\theta}) = 
\left\{
\begin{aligned}
G_n^{(1)}(a,p)& = \frac{1}{a}  
\sum_{k=1}^n \Big(\frac{k^{2p}}{a^2} o_k -2\Big)
= \frac{1}{a}  
\sum_{k=1}^n \Big(
\frac{k^{2(p-p_0)}a_0^2}{a^2} Z_k -2
\Big)
\notag 
\\
G_n^{(2)}(a,p) & = 
\sum_{k=1}^n 
\log(k) \Big( 2
-\frac{k^{2p}}{a^2} o_k\Big) 
= 
\sum_{k=1}^n 
\log(k) \Big( 2
-\frac{k^{2(p-p_0)}a_0^2}{a^2} Z_k \Big) 
\end{aligned}
\right.
\]
and
\[
\mathbf{H}^{-1}_n(a_0,p_0) =
\begin{pmatrix}
\frac{a_0}{2\sqrt{n}}
& 0
\\
0 
& 
\frac{1}{2\sqrt{\sum_{k=1}^n 
\log^2(k)}}
\end{pmatrix}
\]
it is simple to state that
\begin{equation}\label{eq:asymp_G}
\mathbf{H}^{-1}_n(a_0,p_0)\cdot
\mathbf{G}_n(a_0,p_0)
\mathop{\longrightarrow}\limits_{n\to\infty}^L
\begin{pmatrix}
1\\
-1
\end{pmatrix}Z.
\end{equation}
In fact, since $\{Z_k\}_{k\geq 0}$ is a i.i.d.\ sequence of random variables with mean $2$ and variance $4$
(see \eqref{eq:Zk}), we get,
\[
E( G_n^{(1)}(a_0,p_0) G_n^{(2)}(a_0,p_0) ) = 
- E\Big( 
\frac{1}{a_0}  
\sum_{k=1}^n 
\log(k) ( 2
- Z_k )^2
\Big) = -\frac{4}{a_0}  
\sum_{k=1}^n \log(k) 
\]
and hence
\[
Corr\Big(
\frac{a_0}{2\sqrt{n}}
G_n^{(1)}(a_0,p_0) , 
\frac{G_n^{(2)}(a_0,p_0)
}{2\sqrt{\sum_{k=1}^n 
\log^2(k)}}
\Big) =
\frac{ -\sum_{k=1}^n \log(k) }{
\sqrt{n}\sqrt{\sum_{k=1}^n 
\log^2(k)}
} \mathop{\longrightarrow}\limits_{n\to\infty} -1.
\] 
Now, since
\[
\begin{aligned}
\dot{\mathbf{G}}(a_0,p_0)
&= 
\begin{pmatrix}
-\frac{4n}{a_0^2} \Big( 1 +\frac{3}{4}  \frac{\sum_{k=1}^n (Z_k-2)}{n} \Big)
& \frac{4 \sum_{k=1}^n 
\log(k)}{a_0} \Big( 1 +
\frac{\sum_{k=1}^n 
\log(k) (Z_k-2)}{ 2\sum_{k=1}^n \log(k) } \Big)
\\
\frac{4 \sum_{k=1}^n 
\log(k)}{a_0} \Big( 1 +
\frac{\sum_{k=1}^n 
\log(k) (Z_k-2)}{ 2\sum_{k=1}^n \log(k) } \Big)
 & 
-4 \sum_{k=1}^n \log^2(k)
\Big( 1 + \frac{\sum_{k=1}^n 
\log^2(k) (Z_k -2)}{ 2\sum_{k=1}^n \log^2(k) } 
\end{pmatrix}
\\
&  \asymp
\begin{pmatrix}
-\frac{4n}{a_0^2} 
& \frac{4 \sum_{k=1}^n 
\log(k)}{a_0} 
\\
\frac{4 \sum_{k=1}^n 
\log(k)}{a_0} 
 & 
-4 \sum_{k=1}^n \log^2(k)
\end{pmatrix}
\end{aligned}
\]
then, by \eqref{eq:asymp_G} (see \cite[pag.191]{Heyde}), we get
\[
\mathbf{H}^{-1}_n(a_0,p_0) \cdot (- \dot{\mathbf{G}}(a_0,p_0))
\cdot 
\begin{pmatrix}
\hat{a}_n-a_0
\\
\hat{p}_n-p_0
\end{pmatrix}
\mathop{\longrightarrow}\limits^{L}_{n\to\infty}
\begin{pmatrix}
1
\\
-1
\end{pmatrix}
Z,
\]
which is the thesis, once the conditions of uniformly boundedness are checked as for 
the previous case.
\qed\end{proof}


\begin{thebibliography}{18}
\providecommand{\natexlab}[1]{#1}
\providecommand{\url}[1]{{#1}}
\providecommand{\urlprefix}{URL }
\expandafter\ifx\csname urlstyle\endcsname\relax
  \providecommand{\doi}[1]{DOI~\discretionary{}{}{}#1}\else
  \providecommand{\doi}{DOI~\discretionary{}{}{}\begingroup
  \urlstyle{rm}\Url}\fi
\providecommand{\eprint}[2][]{\url{#2}}

\bibitem[{Ash(1990)}]{ash}
Ash RB (1990) \emph{Information theory}. Dover Publications Inc., New York, corrected
  reprint of the 1965 original

\bibitem[{Bass(2011)}]{Bass_SP}
Bass RF (2011) \emph{Stochastic processes}. Cambridge Series in Statistical and
  Probabilistic Mathematics, vol~33. Cambridge University Press, Cambridge

\bibitem[{Brigham(1982)}]{FFT:book}
Brigham EO (1982) \emph{F{FT}: schnelle {F}ourier-{T}ransformation. Einf\"uhrung in
  die Nachrichtentechnik}. [Introduction to Information Technology], R.
  Oldenbourg Verlag, Munich, translated from the English by Seyed Ali Azizi

\bibitem[{Burrough and Frank(1996)}]{Burrough}
Burrough PA, Frank A (1996) \emph{Geographic objects with indeterminate boundaries: 
GISDATA 2 (GISDATA Series)}. Taylor \& Francis, Great Britain

\bibitem[{Dioguardi et~al({2003})Dioguardi, Franceschini, Aletti, Russo, and
  Grizzi}]{DiogAletti}
Dioguardi N, Franceschini B, Aletti G, Russo C, Grizzi F ({2003}) {Fractal
  dimension rectified meter for quantification of liver fibrosis and other
  irregular microscopic objects}. \emph{Analytical and Quantitative Cytology and
  Histology}, {25}({6}):{312--320}

\bibitem[{Dutt and Rokhlin(1993)}]{FFT_nonUniform}
Dutt A, Rokhlin V (1993) Fast {F}ourier transforms for nonequispaced data. 
\emph{SIAM Journal on Scientific Computing}, 14(6):1368--1393, \doi{10.1137/0914081},
  \urlprefix\url{http://dx.doi.org/10.1137/0914081}

\bibitem[{Hall and Heyde(1980)}]{Hall_Heyde}
Hall P, Heyde CC (1980) \emph{Martingale limit theory and its application}. 
Academic  Press Inc. [Harcourt Brace Jovanovich Publishers], New York

\bibitem[{Heyde(1997)}]{Heyde}
Heyde CC (1997) \emph{Quasi-likelihood and its application. A general approach to
  optimal parameter estimation}. Springer Series in
  Statistics, Springer-Verlag, New York, \doi{10.1007/b98823},
  \urlprefix\url{http://dx.doi.org/10.1007/b98823}

\bibitem[{Hobolth(2003)}]{HoboltNeurology}
Hobolth A (2003) The spherical deformation model. \emph{Biostatistics}, 4(4):583--595

\bibitem[{Hobolth and Vedel~Jensen(2002)}]{HoboltStereology}
Hobolth A, Vedel~Jensen E (2002) A note on design-based versus model-based
  estimation in stereology. \emph{Advances in Applied Probability}, 34(1):484--490

\bibitem[{Hobolth et~al(2003)Hobolth, Pedersen, and Vedel~Jensen}]{HoboltShape}
Hobolth A, Pedersen J, Vedel~Jensen E (2003) A continuous parametric shape
  model. \emph{Annals of the Institute of Statistical Mathematics}, 55(2):227--242

\bibitem[{Karhunen(1947)}]{Karhunen}
Karhunen K (1947) \"{U}ber lineare {M}ethoden in der
  {W}ahrscheinlichkeitsrechnung. \emph{Annales Academi\ae\ Scientiarum Fennic\ae. Series A 1, Mathematica - Physica},
  1947(37):79

\bibitem[{Kroese et~al(2011)Kroese, Taimre, and Botev}]{Kroese_HB}
Kroese DP, Taimre T, Botev ZI (2011) \emph{Handbook of Monte Carlo Methods}. John
  Wiley and Sons, New Jersey, U.S.A.

\bibitem[{van Lieshout(2013)}]{Lieshout}
van Lieshout MNM (2013) A spectral mean for point sampled closed curves. \emph{arXiv
  preprint} \url{arXiv:1310.7838}

\bibitem[{Lorentz(1948)}]{Boas}
Lorentz GG (1948) Fourier-{K}oeffizienten und {F}unktionenklassen. \emph{Mathematische Zeitschrift},
  51:135--149

\bibitem[{Manganaro(2011)}]{Manganaro}
Manganaro G (2011) \emph{Advanced Data Converters}. Cambridge University Press,
  Cambridge.

\bibitem[{Patel and Read(1982)}]{normale}
Patel JK, Read CB (1982) \emph{Handbook of the normal distribution}, Statistics:
  Textbooks and Monographs, vol~40. Marcel Dekker Inc., New York

\bibitem[{Revuz and Yor(1999)}]{Revuz}
Revuz D, Yor M (1999) \emph{Continuous martingales and {B}rownian motion}, Grundlehren
  der Mathematischen Wissenschaften [Fundamental Principles of Mathematical
  Sciences], vol 293, 3rd edn. Springer-Verlag, Berlin

\end{thebibliography}
\end{document}